\documentclass[11pt,reqno]{amsart}
\title{Contractibility and total semi-stability conditions of Euclidean quivers}
\author{Yu Qiu}
\author{Xiaoting Zhang}

\newcommand{\Qy}[1]{\textcolor{blue!50!cyan}{#1}}
\usepackage[a4paper]{geometry}
\usepackage{float}
\usepackage[usenames,dvipsnames]{xcolor}
\usepackage{subfigure}

\usepackage[colorlinks, linkcolor=blue,anchorcolor=Periwinkle,
   citecolor=red,urlcolor=Emerald]{hyperref}
\usepackage{amsmath}           % AMS 的主包
\usepackage{amssymb}           % AMS 符号包，会自动加载 amsfonts
\usepackage{latexsym}          % 几个额外的特殊符号，包括 \Box
\usepackage{bookmark}
\usepackage{enumitem,enumerate}
\usepackage{array}
\usepackage{graphicx}
\usepackage[bottom]{footmisc}
\usepackage{cleveref}
\usepackage{transparent}

\usepackage{tikz}
\usetikzlibrary{matrix}
\usepackage{url}
\usetikzlibrary{decorations.markings,cd}
\usetikzlibrary{arrows}
\usetikzlibrary{decorations.pathreplacing,decorations.pathmorphing,shadings,fadings,calc}
\usepackage{extarrows,stmaryrd}

\tikzset{->-/.style={decoration={  markings,  mark=at position #1 with
    {\arrow{>}}},postaction={decorate}}}
\tikzset{-<-/.style={decoration={  markings,  mark=at position #1 with
    {\arrow{<}}},postaction={decorate}}}
\tikzcdset{arrow style=tikz, diagrams={>=stealth}}

\newcolumntype{L}{>{$}l<{$}} % math-mode version of "l" column type

\usepackage{tikz}
\usetikzlibrary{matrix, calc, arrows,cd,
decorations.pathreplacing, decorations.markings, decorations.pathmorphing,
backgrounds,fit,positioning,shapes.symbols,chains,shadings,fadings}

\usepackage{mathabx}

\usepackage{marvosym}

\theoremstyle{plain}
\newtheorem{theorem}{Theorem}[section]
\newtheorem{lemma}[theorem]{Lemma}
\newtheorem{corollary}[theorem]{Corollary}
\newtheorem{proposition}[theorem]{Proposition}

\newtheorem{convention}[theorem]{Convention}

\theoremstyle{definition}
\newtheorem{definition}[theorem]{Definition}

\numberwithin{equation}{section}
\newcommand{\on}[1]{\operatorname{#1}}

\newcommand\hua{\mathcal}

\newcommand\ZZ{\mathbb{Z}}

\newcommand\RR{\mathbb{R}}
\newcommand\CC{\mathbb{C}}
\newcommand\PP{\mathbb{P}}
\newcommand\wt{\mathbf{w}}
\newcommand\<{\langle}
\renewcommand\>{\rangle}

\newcommand{\rank}{\on{rank}}

\newcommand{\Proj}{\on{Proj}}
\newcommand{\Inj}{\on{Inj}}

\newcommand\Hom{\on{Hom}}

\newcommand{\h}{\on{\hua{H}}} %heart
\newcommand{\D}{\on{\hua{D}}} % triangulated category
\newcommand\Stab{\on{Stab}} % space of stability conditions
\newcommand\PStab{\on{\PP Stab}} % space of stability conditions
\newcommand{\uhp}{\mathbb{H}} % upper half plane
\newcommand{\sli}{\hua{P}}

\newcommand{\wh}[1]{\widetilde{#1}}

\usepackage{pifont}%\Ding

\def\be{\begin{equation}}
\def\ee{\end{equation}}

\def\ToSS{\on{ToSS}}
\def\PToSS{\on{\PP ToSS}}
\def\canA{\on{C}_\wt}
\def\DQ{\D^b(\wpl)}
\def\wpl{\PP^1_\wt}
\def\cohW{\on{coh}(\wpl)}

\def\tube{\mathcal{T}}

\def\Apq{\wh{A_{p,q}}}
\def\ADn{\wh{D_{n}}}
\def\AEn{\wh{E_{n}}}

\def\buhp{\CC^*_{\ge0}}

\newcommand\OO{\mathcal{O}}
\renewcommand{\vec}[1]{ \overrightarrow{#1} }
\def\Vecc{\vec{c}}
\def\Vecw{\vec{\omega}}

\def\EQ{\wh{\Omega}}

\def\SEC{\Gamma}
\def\TSD{\on{TSD}}
\def\PPwt{\PP^1\setminus\wt}

\def\udim{\underline{\on{dim}}}

\usepackage{pdflscape}
\def\nw{l}

\begin{document}
%=========================================================
\begin{abstract} %Notation free version:
We study the bounded derived category $\D$ of an Euclidean quiver, or equivalently, that
of coherent sheaves on a tame weighted projective line.
We give a description of the moduli space $\ToSS$ of the total semi-stability conditions on $\D$, which implies that $\ToSS$ can linearly contract to
any chosen non-concentrated stability condition in it.
For type $\Apq$, this gives an alternative proof of the contractibility of the whole space of stability conditions.

\end{abstract}
\maketitle

\tableofcontents \addtocontents{toc}{\setcounter{tocdepth}{1}}
\setlength\parindent{0pt}
\setlength{\parskip}{5pt}

%=========================================================
%=========================================================
\section{Introduction}
%=========================================================
\subsection{Stability conditions}
%=========================================================
The spaces of stability conditions were introduced by Bridgeland \cite{B1} two decades ago,
building on the classical stability in geometric invariant theory, King's $\theta$-stability and Douglas' $\Pi$-stability.
Since then, there are tons of exciting developments and applications,
that relates to Donaldson-Thomas theory, cluster theory, flat surfaces, etc.
Despite all the works, many questions and conjectures remain to be explored.

We are in particular interested in the contractibility conjecture of the stability spaces.
Such a conjecture is related to the classical $K(\pi,1)$-conjecture, by realizing certain hyperplane arrangements as quotient of the stability spaces of some Calabi-Yau categories, cf. \cite{B1, B2} and \cite{QW} for instance.
Thus, the contractibility conjecture can be viewed as a categorical $K(\pi,1)$-conjecture,
which holds in greater generality (beyond hyperplane arrangements).

%\begin{theorem}\label{thm:0}
%Let $Q$ be an acyclic quiver. Then we have the following characterization of finiteness/tameness/wildness:
%\begin{itemize}
%  \item $Q$ is Dynkin if and only if there ia a total stability condition.
%  \item $Q$ is Euclidean if and only if, there ia no total stability condition but there is a non-concentrated total semi-stability condition.
%  \item $Q$ is wild if and only if any total semi-stability condition is non-concentrated.
%\end{itemize}
%\end{theorem}

%=========================================================
\subsection{Total semi-stability}
%=========================================================

In \cite{Q3}, we introduce a piecewise Morse-ish function,
the global dimension function $\on{gldim}$ on the stability space $\Stab\D$ of a triangulated category $\D$.
This opens up a strategy to study the contractibility conjecture for non-Calabi-Yau categories.
The infimum of such a function, denote by $\on{gd}(\D)$, should be considered as the global dimension of $\D$.
A stability-version of Gabriel theorem is that, if $\on{gd}(\D)<1$ if and only if $\D=\D^b(\Omega)$ is the bounded derived category of a Dynkin quiver $\Omega$ (\cite[Thm.~3.2]{Q3}), which is rare.
The inequality $\on{gldim}(\sigma)<1$ for a stability condition $\sigma$ is equivalent to the condition that all indecomposable objects in $\D$ is $\sigma$-stable.
In which case, we call $\sigma$ is totally stable.
The leads to the study of total stability conditions, which is conducted in \cite{QZ}.

The next case, following this line of work, is to study total semi-stability conditions,
that corresponds to $\on{gldim}(\sigma)\le1$, cf. \cite[\S~3.2]{Q2}.
In this paper, we focus on the case for $\D=\D^b(\EQ)$ for an Euclidean quiver $\EQ$, which is derived equivalent to $\DQ$,
the bounded derived category of coherent sheaves on a weighted projective line $\wpl$ of tame type.

Let $\wt=(w_1,\ldots,w_l)$ be the weight of $\wpl$.
A \emph{totally semi-stable datum} $\TSD=(\mu_?^?,z)$ consists of a set of partitions
\begin{gather}
    1=\mu_i^1+\cdots+\mu_i^{w_i},\quad 1\le i\le \nw,
\end{gather}
for real positive numbers $\mu_i^j$, and $z\in\buhp$.
It is \emph{non-degenerate} if the central charge $Z$ on $\DQ$ determined by \eqref{eq:central-c}
is non-degenerate (i.e. for any $V\in\on{vect}(\wpl)$, one has $Z([V])\ne0$).

Moreover, we say a stability condition $\sigma$ is \emph{concentrated} if its heart $\h_\sigma$ is concentrated in one slice, i.e. $\h_\sigma=\sli(\phi)$ for some $\phi\in(0,1]$.

Here are the main results of the paper.

\begin{theorem}\label{thm:0}
%Let $\DQ$ be the bounded derived category of the coherent sheaves on a weighted projective line $\wpl$ of tame type.
Let $\PToSS(\wpl)$ be the projective space of total semi-stable stability conditions on $\DQ$.
\begin{itemize}
  \item A non-degenerate totally semi-stable datum $\TSD=(\mu_?^?,z)$ determines a stability condition on $\DQ$ by \eqref{eq:central-c}.
  It is totally semi-stable (i.e. in $\PToSS(\wpl)$) if it further satisfies a set of inequalities:
  \eqref{eq:D} for affine type D, \eqref{eq:AR-E6}/\eqref{eq:AR-E7}/\eqref{eq:AR-E8} for affine type E, respectively.
  \item For a representative $\sigma$ in $\PToSS(\wpl)$ such that any object in any rank one tube has phase one,
  we have
    \begin{itemize}
      \item $\sigma$ is non-concentrated if and only if $\h_\sigma=\cohW$.
      \item $\sigma$ is concentrated if and only if $\h_\sigma\cong\on{mod}\CC\EQ$ for some Euclidean quiver $\EQ$.
    \end{itemize}
  \item Given any non-concentrated stability condition $\sigma_0$ in $\PToSS(\wpl)$ (which exists),
  there is a linear flow that contracts $\PToSS(\wpl)$ to $\sigma_0$.
\end{itemize}
\end{theorem}

The point is that we translate the semi-stability to a set of inequalities (in term of the partitions of the central charge of the imagery root), which are linear inequalities of central changes of exceptional simples of $\wpl$. These allow us to contract $\PToSS(\wpl)$ linearly.

An application of the study above is the following contractibility theorem for affine type A,
where we import the partial contractibility result from \cite{Q3}, that $\Stab\D^b(\Apq)$ contracts to $\ToSS\D^b(\Apq)$.

\begin{corollary}
The space of stability conditions $\Stab\D^b(\Apq)$ is contractible.
\end{corollary}

Such a theorem was proved in \cite{HKK} by showing $\Stab\D^b(\Apq)$ is isomorphic to the moduli space of quadratic differentials on $\CC\PP^1$ with an exceptional singularity.
Remarks on previous/related works:
\begin{itemize}
  \item The contractibility for $\wh{A_{1,1}}$ case
  is due to the calculation of the space in \cite{Oka}, cf. \cite{Q2}.
  \item The contractibility for $\wh{A_{1,2}}$ case is due to \cite{DK1}.
  \item The contractibility for the wild Kronecker quiver case
  is due to the calculation of the space in \cite{DK2}.
  \item The contractibility for the 3-Calbi-Yau $\Apq$ case
  is due to the calculation of the space in \cite{QQ}.
\end{itemize}

We think our approach of such a contractibility could apply to affine type D/E as well.
%=========================================================
%\paragraph{\textbf{Related works}}\

%=========================================================
\subsection*{Acknowledgments}
%=========================================================
We would like to thank 
He Ping and Han Zhe for proofreading;
Bernhard Keller and Chen Xiaowu for pointing out useful references, e.g. \cite{DR,Rin},
and Otani Takumi for pointing out his related work \cite{Ota}.
This work is supported by
National Natural Science Foundation of China (Grant No. 12425104, 12031007, 12101422 and 12271279)
and National Key R\&D Program of China (No. 2020YFA0713000).

%=========================================================
%=========================================================
\section{Preliminaries}
%=========================================================
We will work over the field $\CC$ for simplicity.
Let $\D$ be a triangulated category whose Grothendieck group
$K(\D)$ is a finite rank lattice, i.e. $K(\D)\cong\ZZ^n$ for some $n$.

Let $\CC_{>0}=\{z\in\CC\mid\on{Im}z>0\}$ be the strict upper half plane,
$\uhp=\CC_{>0}\bigcup\RR_{<0}$ the half-open-half-closed upper half plane and
$\buhp=\{z\in\CC^*\mid\on{Im}z\ge0\}$ be the whole upper half plane.
%=========================================================
\subsection{Stability conditions on triangulated categories}\
%=========================================================

Recall the notion of stability conditions on a triangulated category from \cite{B1}.

\begin{definition}
\label{def:stab}
A {\it stability condition} $\sigma = (Z, \sli=\sli_\RR)$ on $\D$ consists of
a group homomorphism $Z \colon K(\D) \to \CC$, known as the {\it central charge}, and
an $\RR$-family of full additive (abelian in fact) subcategories $\sli(\phi) \subset \D$, $\phi\in\RR$,
known as the {\it slicing},
satisfying the following conditions:
\begin{itemize}
\item if  $0 \neq E \in \sli(\phi)$,
then $Z(E) = m(E) \exp(\mathbf{i} \pi \phi)$ for some $m(E) \in \RR_{>0}$,
\item
for all $\phi \in \RR$, $\sli(\phi + 1) = \sli(\phi)[1]$,
\item if $\phi_1 > \phi_2$ and $A_i \in \sli(\phi_i)\,(i =1,2)$,
then $\Hom(A_1,A_2) = 0$,
\item for each object $0 \neq E \in \D$, there is a finite sequence of real numbers
\begin{equation}\label{eq:>}
\phi_1 > \phi_2 > \cdots > \phi_l
\end{equation}
and a collection of exact triangles (known as the \emph{HN-filtration})
\begin{equation*}
\begin{tikzcd}[column sep=.8pc]
    0=E_0 \ar[rr] && E_1 \ar[rr]\ar[dl] && E_2 \ar[r]\ar[dl] & \ldots \ar[r] & E_{l-1} \ar[rr] && E_l=E \ar[dl] \\
        &A_1 \ar[ul,dashed]&&A_2 \ar[ul,dashed]&&&& A_l \ar[ul,dashed]
\end{tikzcd}
\end{equation*}
with $A_i \in \sli(\phi_i)$ for all $i$.
\item a technique condition, known as the \emph{support property}, which comes free in our setting.
\end{itemize}
\end{definition}
The non zero objects in slices $\sli(\phi)$ are called {\it semistable with phase $\phi$}
and simple objects among them are called {\it stable with phase $\phi$}.
%Finally, for any interval $J\subset\RR$, denote by $\sli(J)$ the subcategories consisting of
%objects whose HN-filtrations only have factors with phases in $J$.
For a semistable object $E\in\sli(\phi)$, denote by $\phi_\sigma(E)\coloneq\phi$ its phase.

There is a natural $\CC$-action on the set $\Stab\D$ of all stability conditions on $\D$, namely:
\[
    s \cdot (Z,\sli)=(Z \cdot e^{-\mathbf{i} \pi s},\sli_{\on{Re}(s)}),
\]
where $s\in\CC$ and $\sli_{x}(\phi)=\sli(\phi+x)$ for $\forall x,\phi\in\RR$.
Denote by $$\PStab(\D)=\Stab(\D)/\CC$$ the space of projective stability conditions.

The famous result in \cite{B1} states that $\Stab\D$ is a complex manifold with dimension $\rank K(\D)$
and the local coordinate is provided by the central charge $Z$.

\begin{definition}
For each $\phi\in\RR$, there is a heart $\h_\sigma^\phi=\sli(\phi,\phi+1]$.
We call $\h_\sigma=\h_\sigma^0$ the heart of $\sigma$.

We say a stability condition $\sigma$ is \emph{concentrated} if its heart $\h_\sigma$ is concentrated in one slice, i.e. $\h_\sigma=\sli(\phi)$ for some $\phi\in(0,1]$.
\end{definition}

%=========================================================
\subsection{Total semi-stability}\
%=========================================================

\begin{definition}
A \emph{total (semi-)stability condition} on $\D$ is a stability condition $\sigma$ such that any indecomposable object of $\D$ is (semi-)stable.
\end{definition}

Recall some facts.
\begin{itemize}
  \item There is a continuous function $\on{gldim}$, called \emph{global dimension function},
  on $\PStab(\D)$, defined as (cf. \cite{IQ,Q2})
    \begin{gather}\label{eq:geq}
      \on{gldim}\sigma=\sup\{ \phi_2-\phi_1 \mid
        \Hom(\sli(\phi_1),\sli(\phi_2))\neq0\}\in\RR_{\ge0}\cup\{+\infty\}.
    \end{gather}
  \item The \emph{global dimension} $\on{gldim}(\D)$ of a triangulated category $\D$ is defined to be the infimum of the function $\on{gldim}$ on $\PStab(\D)$ (cf. \cite{IQ,Q2}).
  \item A sufficient condition for a stability condition $\sigma$ to be \emph{totally stable} is
    $\on{gldim}(\D)<1$, which is equivalent to the fact that $\D$ is the derived category associated to a Dynkin diagram (of type A/B/C/D/E/F/G, cf. \cite[Thm.~3.2]{Q3}).
  \item There is a classification/description of (moduli space of) total stability conditions of any Dynkin diagram (\cite{QZ}).
  \item A stability condition $\sigma$ is \emph{totally semi-stable} if and only if $\on{gldim}\sigma\le1$ (\cite[Prop.~3.5]{Q2}).
\end{itemize}

An immediate corollary of \cite[Prop.~3.5]{Q2} is the following.
\begin{lemma}
A concentrated stability condition is totally semi-stable if and only if the heart is hereditary.
\end{lemma}
\begin{proof}
Since the stability condition $\sigma$ is concentrated, we have $\on{gldim}\sigma=\on{gldim}\h_\sigma$.
Then, the total semi-stability is equivalent to $\on{gldim}\h_\sigma\le1$, namely, $\h_\sigma$ is hereditary.
\end{proof}

In this work, our first aim is to describe the subspace
$$
    \ToSS(\D)=\on{gldim}^{-1}(0,1]\subset\Stab(\D)
$$
consisting of total semi-stability conditions on the triangulated category $\D^b(\EQ)$ for an Euclidean quiver $\EQ$.

We will frequently use the following easy fact.
\begin{lemma}\label{lem:easy}
Let $\sigma$ be totally semi-stable.
If there is a (directly) path from an (indecomposable) object $M$ to $N$ in the Auslander-Reiten quiver,
then $\phi_\sigma(M)\le\phi_\sigma(N)$.
In particular, $\phi_\sigma(\tau M)\le\phi_\sigma(N)$.
\end{lemma}

\def\REF{R}
\begin{convention}\label{con1}
We may work with the projective version $\PToSS(\D)=\ToSS(\D)/\CC$.
Moreover, we will fix an indecomposable object $\REF$ in $\D$ as a reference object
and always normalize a stability condition such that $Z(\REF)=-1$ with $\phi_\sigma(\REF)=1$
to represent a point in $\PToSS(\D)$.
\end{convention}
%=========================================================
\subsection{Weighted projective lines of tame type and canonical algebras}\
%=========================================================

Consider the weighted projective line $\CC\PP^1(w_1,w_2,\ldots,w_\nw)$ of tame type, where each $w_i (1\le i\le l)$ 
 is an integer. Write $\wt=(w_1,\ldots,w_\nw)$ and $\wpl$ for $\CC\PP^1(\wt)$.
Denote by
\begin{itemize}
  \item $\cohW$ the abelian category of coherent sheaves on $\wpl$,
  \item $\on{vect}(\wpl)$ its subcategory consisting of vector bundles on $\wpl$,
  \item $\on{coh}_0(\wpl)$ its subcategory consisting of finite-length/torsion objects and
  \item $\DQ=\D^b(\cohW)$ its bounded derived category.
\end{itemize}

The rank one abelian group $\mathbb{L}$ is defined by
\[
    \mathbb{L}\colon=\<\vec{x_1},\ldots,\vec{x_\nw},\Vecc \mid w_i\,\vec{x_i}=\Vecc \>
\]
Denote by $\OO(\vec{x})$ the \emph{twisted structure sheaf} on $\wpl$ for any $\vec{x}\in\mathbb{L}$.
For more details, see \cite{GL}.

There is a \emph{canonical algebra} $\on{C}_\wt= \CC Q_\wt^{\on{op}}/I_\wt$ (cf. \cite{Rin}) with the quiver $Q_\wt$
\begin{equation}\label{eq:pqr}
\begin{tikzcd}[column sep=25,row sep=18]
    & V_1^1 \ar[r,"\Qy{x_1^2}"] & \cdots\cdots \ar[r,"\Qy{x_1^{w_1-1}}"]
        & V_1^{w_1-1} \ar[dr,bend left=10,"\Qy{x_1^{w_1}}"]\\
    V_0 \ar[ur,bend left=10,"\Qy{x_1^{1}}"] \ar[ddr,bend left=-20,"\Qy{x_\nw^{1}}"'] \ar[r,"\Qy{x_2^{1}}"']
        & V_2^{1} \ar[r,"\Qy{x_2^{2}}"'] & \cdots\cdots \ar[r,"\Qy{x_2^{w_2-1}}"']
        & V_2^{w_2-1}\ar[r,"\Qy{x_2^{w_2}}"'] & V_{\infty} \\
    &&\cdots\cdots\\
    & V_\nw^{1} \ar[r,"\Qy{x_\nw^{2}}"'] & \cdots\cdots \ar[r,"\Qy{x_\nw^{w_\nw-1}}"']
        & V_\nw^{w_\nw-1} \ar[uur,bend left=-20,"\Qy{x_\nw^{w_\nw}}"']
\end{tikzcd}
\end{equation}
and the canonical relation $I_\wt$:
\[
    \prod_{j=1}^{w_1} x_1^j+\prod_{j=1}^{w_2} x_2^j+
        \cdots+\prod_{j=1}^{w_\nw} x_\nw^j=0, \quad\text{if $\nw\ge3$}.
\]
%(among other relations, which we don't need when restricted to the tame case).
We set $V_i^0=V_0$ and $V_i^{w_i}=V_\infty$ for any $1\le i\le \nw$.
Then there is a triangle equivalence (\cite{GL})
\begin{equation}\label{eq:tri-equiv}
  \D^b(\canA)\cong\DQ
\end{equation}
identifying
\[\begin{cases}
    P_0=\OO(\vec{0}), \\
    P_i^j=\OO(j\,\vec{x_i}), & \mbox{for $1\le i\le \nw$, $1\le j<w_i$},\\
    P_\infty=\OO(\Vecc).
  \end{cases}
\]
Here $P_?^?$ is the projective corresponding to the vertex $V_?^?$.
Set $P_i^0=P_0$ and $P_i^{w_i}=P_\infty$ for any $1\le i\le \nw$.
The Auslander-Reiten functor $\tau$ is realized by
\[
    \OO(\vec{v}+\Vecw)=\tau( \OO(\vec{v}) )
\]
on twisted structure sheaves, where
\[
    \Vecw=(\nw-2)\Vecc-\sum_{i=1}^{\nw}\vec{x_i}
\]
is the dualizing element.

We will only consider the tame case, i.e. there is an Euclidean quiver $\EQ$ such that
there is another triangle equivalence
\begin{equation}\label{eq:tri-equiv2}
  \DQ\cong\D^b(\EQ).
\end{equation}
More precisely, we have
\begin{itemize}
  \item For type $\Apq$: $\nw=2$ and $(w_1,w_2)=(p,q)$.
  However, we may treat this case as $\nw=3$ with $(w_1,w_2,w_3)=(p,q,1)$ uniformly, where the $w_3=1$ branch in \eqref{eq:pqr} vanishes/degernates.
  \item For type $\ADn$: $\nw=3$ and $(w_1,w_2,w_3)=(2,2,n-2)$.
  \item For type $\AEn$: $\nw=3$ and $(w_1,w_2,w_3)=(2,3,n-3)$ for $n=6,7,8$.
\end{itemize}
%We will usually write $(p,q,r)$ for $(w_1,w_2,w_3)$.

We will abbreviate $\ToSS(\DQ)$ by $\ToSS(\wpl)$ or $\ToSS(\EQ)$ and similar for $\PToSS$.
%=========================================================
\paragraph{\textbf{Auslander-Reiten (AR) quiver}}\

Note that $\DQ$ is hereditary and any object of which is a shift of objects in $\cohW$.
Moreover, we have
\[
    \on{AR}\DQ=\bigsqcup_{k\in\ZZ} \on{AR}\cohW[k],
\]
where $\on{AR}$ denotes the AR-quiver. Finally,
 $\on{AR}\cohW$ naturally decomposes into:
\[
    \on{AR}\cohW=\on{AR}\on{vect}(\wpl)\bigsqcup \on{AR}\on{coh}_0(\wpl)
\]
where $\on{AR}\on{vect}(\wpl)\cong\ZZ\EQ$ consists of indecomposable vector bundles and
$\on{AR}\on{coh}_0(\wpl)$ consisting of a $\CC$-family of tubes.

We have not explicitly specify the triangle equivalence, which we will do later in each case.
Nevertheless, we will choose a section $\SEC$ in $\on{AR}\on{vect}(\wpl)$
and identify it with $\on{AR}\Proj\CC\EQ\cong\EQ^{\on{op}}$.

%=========================================================
\paragraph{\textbf{Tubes in $\on{AR}\on{coh}_0(\wpl)$}}\

There are two types of tubes:
\begin{itemize}
  \item A rank $w_i$ tube $\tube_i$ consisting of \emph{exceptional simple} $S_i^j$ in its bottom, for $1\le i\le \nw$:
    \begin{equation}\label{eq:tube-i}
      \begin{tikzcd}
        S_i^{w_i} & S_i^{w_i-1} \ar[l,dashed,"\tau"'] & \cdots\ar[l,dashed,"\tau"']
        & S_i^{2}   \ar[l,dashed,"\tau"']
        & S_i^1     \ar[l,dashed,"\tau"'] & S_i^0=S_i^{w_i} \ar[l,dashed,"\tau"'].
      \end{tikzcd}
    \end{equation}
    Here, for $S_i^j$, the superscript $j$ will take values in $\ZZ_{w_i}$ from now on.
  \item A family rank $1$ tubes $\{ \tube^\lambda \}$ with \emph{usual simple} $S^\lambda$, for $\lambda\in\PPwt$.
\end{itemize}

%=========================================================
\paragraph{\textbf{Partitions of the imaginary root}}\

Denote by $\delta\in K( \D^b(\EQ) )$ the \emph{imaginary root} of $\EQ$.
We have the following well-known facts:
\begin{itemize}
\item For each tube $\tube_i$, there is a partition
\begin{equation}\label{eq:sum-p}
    \delta=[S_i^{1}]+\cdots+[S_i^{w_i}].
\end{equation}
\item For each tube $\tube^\lambda$, one has $[S^\lambda]=\delta$.
  \item $[\OO(\Vecc)]-[\OO(\vec{0})]=[P_\infty]-[P_0]=\delta$.
  \item $[\OO(j\,\vec{x_i})]-[\OO( (j-1)\,\vec{x_i})]=[S_i^j]$.
\end{itemize}

%=========================================================
%=========================================================
\section{Total semi-stability data: the prototype}
%=========================================================

%=========================================================
\subsection{The standard heart}
%\paragraph{\textbf{Canonical heart}}\

We will fix $S^{\lambda_0}$ (for some $\lambda_0\in\PPwt$) to be the reference object $\REF$ in \Cref{con1}.
So that for any $\sigma=(Z,\sli)\in\PToSS(\wpl)$, we have
\[Z([S^{\lambda_0}])=Z(\delta)=-1\quad \text{and}\quad\phi_\sigma(S^{\lambda_0})=1.\]

Given a semi-stable stability condition $\sigma$ in $\PToSS(\wpl)$,
we denote
\begin{equation}\label{eq:muij}
    \mu_i^j\colon=-Z([S_i^j])
\end{equation}
for $1\le i\le \nw$ and $j\in\ZZ_{w_i}$.

\begin{lemma}\label{pp:part}
For any usual simple $S^\lambda$ in the bottom of a rank one tube $\tube^\lambda$,
we have $\phi_\sigma(S^\lambda)=1$.
Moreover, $\phi_\sigma(S_i^j)=1$ for all $1\le i\le \nw, j\in\ZZ_{w_i}$ and
\eqref{eq:sum-p} gives $\nw$ partitions of $|Z(\delta)|=1,$ that is,
\begin{gather}\label{eq:part}
    1=\mu_i^1+\cdots+\mu_i^{w_i},\quad 1\le i\le \nw,
\end{gather}
with each $\mu_i^j\in\RR_+$.
\end{lemma}
\begin{proof}
For any usual simple $S^\lambda$, we have
$Z([S^\lambda])=Z(\delta)=-1.$
Since $\sigma\in\PToSS(\wpl)$, then
each $S^\lambda$ is semi-stable and thus $\phi_\sigma(S^\lambda)=1+2k_{\lambda}$ for some $k_{\lambda}\in\ZZ$.

Suppose $\phi_\sigma(S^\lambda)\ne1$ for some $\lambda\in\PPwt$, that is, $k_{\lambda}\neq0$.
Without loss of generality, we can assume $k_{\lambda}<0$, which implies $\phi_\sigma(S^\lambda)\le-1$.
Note that for any indecomposable vector bundle $V$ we have 
\[
    \Hom( S^{\lambda_0}[-1], V)\ne 0 \ne \Hom(V, S^{\lambda} ),
\]
cf. \cite{GL}.
Then the total semi-stability yields that
\[
    0=\phi_\sigma(S^{\lambda_0})-1=\phi_\sigma(S^{\lambda_0}[-1])\le \phi_\sigma(V) \le \phi_\sigma(S^{\lambda})\le -1,
\]
which is a contradiction. Hence $\phi_\sigma(S^\lambda)=1$ for any $\lambda\in\PPwt$.
%{\color{red}Why the other case($k>0$) is similar?}

For any fixed $1\le i\le \nw$, \Cref{lem:easy} implies that
\[
    \phi_\sigma(S_i^{w_i})\le\phi_\sigma(S_i^{w_i-1})\le\cdots\le\phi_\sigma(S_i^1)\le\phi_\sigma(S_i^0)=\phi_\sigma(S_i^{w_i}).
\]
Thus each $S_i^j$, where $j\in\ZZ_{w_i}$, has the same phase, denoted by $\phi_i$. Since \eqref{eq:sum-p} deduces that
\[-1=Z(\delta)=Z([S_i^{1}])+\cdots+Z([S_i^{w_i}]),\] 
then we have $Z([S_i^j])\in\RR_{<0}$ with phase $\phi_i=1+2k_i$ for some $k_i\in\ZZ$.
In particular, \eqref{eq:part} holds.
Using the same argument above, one can obtain that each $k_i$ must be zero.
This completes the proof.
\end{proof}

\begin{proposition}\label{pp:heart}
For $\sigma=(Z,\sli)\in\PToSS(\wpl)$, we have
\begin{itemize}
  \item $\sigma$ is non-concentrated if and only if $\h_\sigma=\cohW$.
  \item $\sigma$ is concentrated if and only if
  $\h_\sigma\cong\on{mod}\CC \EQ'$ for some Euclidean quiver $\EQ'$ if and only if $\cohW<\h_\sigma<\cohW[1]$.
\end{itemize}
Recall that the partial order on hearts is defined to be the reverse inclusion order of the corresponding t-structures.
\end{proposition}
\begin{proof}
As in the proof of Lemma~\ref{pp:part}, we found that the phase of any indecomposable vector bundle lies in $[0,1]$. Denote by $z:=Z([P_0])$. (Recall that $P_0$ is the projective corresponding to the vertex $V_0$ in \eqref{eq:pqr}.)

\paragraph{\textbf{Case I}}
If $\on{Im}z>0$, i.e. $\phi_\sigma(P_0)\in(0,1)$, we claim that
$
    \sli(0,1)=\on{vect}(\wpl),
$
or equivalently, the phase of any indecomposable vector bundle $V$ lies in $(0,1)$.

Note that $K(\DQ)$ admits a basis
\[
    \{ [P_0], \delta=[S^{\lambda_0}] \} \bigsqcup \{ [S_i^j] \mid  1\le i\le l, 1\le j\le w_i-1\}.
\]
Then, for any indecomposable vector bundle $V$, we have
\begin{gather}\label{eq:formula-Z}
    [V]
        =c_0 [P_0] + c_\delta \delta + \sum_{i=1}^{l} \sum_{j=1}^{w_i-1} c_i^j [S_i^j]
\end{gather}
for some $c_0,c_\delta,c_i^j\in\ZZ$.
It implies that
\begin{equation}\label{eq:ImZ}
  \begin{cases}
    \on{Im} Z([V])= c_0\on{Im}(z),\\
    \on{Re} Z([V])
        =c_0 \on{Re}(z) - c_\delta - \sum_{i=1}^{l} \sum_{j=1}^{w_i-1} c_i^j \mu_i^j.
  \end{cases}
\end{equation}
Moreover, we can replace $[P_0]$ in the basis with $[V]$ for any chosen $V$. Thus we have $c_0\ne0$.
As $\phi_\sigma(V)\in[0,1]$, we deduce that $c_0>0$ and hence $\phi_\sigma(V)\in(0,1)$ as claimed.
Therefore,
\[
    \h_\sigma=\sli(0,1]=\< \sli(0,1),\sli(1)\>=\< \on{vect}(\wpl), \on{coh}_0(\wpl) \>=\cohW.
\]

\paragraph{\textbf{Case II}}
If $\on{Im}z=0$, then $\on{Im}Z([V])=0$, by \eqref{eq:ImZ},
for any indecomposable vector bundle $V$.
Thus we have $Z([V])\in\RR^*$ and either
$$\phi_\sigma(V)=0\,(\Leftrightarrow Z ([V])>0)
    \quad\text{or}\quad \phi_\sigma(V)=1\,(\Leftrightarrow Z([V])<0).$$
By \Cref{lem:easy}, $\phi_\sigma(\tau V)\le \phi_\sigma( V)$. Then 
we have either
\[\phi_\sigma(\tau V)=\phi_\sigma(V)\in\{0,1\} \quad\text{or}\quad 0=\phi_\sigma(\tau V)<\phi_\sigma( V)=1.\]
Considering each $\tau$-orbit $\tau^\ZZ V$ in $\on{AR} \on{vect}(\wpl)$, where $V$ is an indecomposable vector bundle $V$, we know that there exists some $j\in\ZZ$ such that
\[
    \ldots=\phi_\sigma(\tau^{j+2} V)=\phi_\sigma(\tau^{j+1} V)=0<1=\phi_\sigma(\tau^{j} V)=\phi_\sigma(\tau^{j-1} V)=\ldots.
\]
Take all such $\tau^j V$'s for all $\tau$-orbits which we claim form a section $\SEC$ of $\on{vect}(\wpl)\cong\ZZ\EQ$.
Again by \Cref{lem:easy}, we can deduce that for the neighbour $\tau$-orbits,
the corresponding $\tau^j V$'s must be nearby (i.e. connected by an arrow).
Then we have
\[
  \h_\sigma=\sli(0,1]=\sli(1)=\<\tau^{\le0}\SEC, \on{coh}_0(\wpl), \tau^{>0}\SEC[1]  \>  \cong
    \on{mod}\CC \SEC^{\on{op}}
\]
where $\SEC^{\on{op}}$ is the opposite quiver of $\SEC$.
Note that $\cohW<\h_\sigma<\cohW[1]$ in this case.

Combining the discussion above, the proposition follows.
\end{proof}
The first case above is also showed in \cite[Prop.~3.10]{Ota}.

%=========================================================
\subsection{Total semi-stability datum}\

\begin{definition}\label{def:datum}
We call a set of partitions \eqref{eq:part} together with $z\in\buhp$
a \emph{totally semi-stable datum}, denoted by $\TSD=(\mu_?^?,z)$.
It determines a central charge $Z$ on $\DQ$ by
\begin{equation}\label{eq:central-c}
\begin{cases}
  Z([S_i^j])=-\mu_i^j,& \quad 1\le i\le \nw,j\in\ZZ_{w_i}. \\
  Z([P_0])=z.\end{cases}
\end{equation}
Such a datum $\TSD$ is \emph{non-degenerate} if
for any indecomposable vector bundle $V\in\on{vect}(\wpl)$, one has $Z([V])\ne0$.
\end{definition}
Note that, by \eqref{eq:ImZ}, the condition $\on{Im}(z)>0$ implies that the totally semi-stable datum $\TSD=(\mu_?^?,z)$ is non-degenerate.

By Lemma~\ref{pp:part}, a stability condition $\sigma\in\PToSS(\wpl)$ determines a non-degenerate totally semi-stable datum $\TSD$. However, the converse statement is not true in general.

%{\color{red}Using \eqref{eq:formula-Z}, one can deduce some inequalities for a non-degenerate totally semi-stable datum $\TSD$. For instance, in type $\Apq$, the non-degeneracy is the following family of inequalities:
%\begin{equation}
%    z\ne k+\mu_1^1+\cdots+\mu_1^{k_1}+\mu_2^1+\cdots+\mu_2^{k_2},
%\end{equation}
%for any $k\in\ZZ, 0\le k_1<p$ and $0\le k_2<q$. Why $\on{Im}(z)>0$ is obvious?}

%\begin{remark}\label{rem:determie}
Let $\arg$ be the function on $\buhp$ with codomain $[0,\pi]$.
A non-degenerate totally semi-stable datum $\TSD$ also determines a slicing $\sli$ by letting
all indecomposable object $V[m]$ lie in the slice $\sli(\phi+m)$ for any $V\in\cohW, m\in\ZZ$,
where $\phi=\pi^{-1}\arg Z(V)$.

Similar to the total stability case (cf. \cite{QZ}), we have the following criterion of total semi-stability.
\begin{proposition}\label{pp:AR}
A non-degenerated totally semi-stable datum $\TSD$ induces a total semi-stability condition $\sigma=(Z,\sli)$ as above if and only if
\begin{itemize}
  \item[$\bigstar$] for every arrow $X \to Y$ in the %(shift of the) 
    preprojective (or vector bundle) component $\on{AR}\on{vect}(\CC\PP^1_\wt)$ of the AR-quiver of $\DQ$, 
    one has $\phi_\sigma(X)\le\phi_\sigma(Y)$.
\end{itemize}
\end{proposition}
\begin{proof}
The `only if' part is straightforward. 
For the `if' part, we will discuss in the following two cases.

If $\on{Im}(z)>0$, 
the central charge $Z$ is in fact a stability function on the heart $\h_\sigma=\cohW$ by \eqref{eq:ImZ}.
By definition, an (indecomposable) object $E$ in $\cohW$ is semi-stable if and only if the phase of
any its submodule is less or equal than its phase. If $E$ is a vector bundle, Condition~$\bigstar$
implies that it must be semi-stable. If $E$ is in a tube, it has the biggest phase (i.e. one)
and hence is also semi-stable. Thus $Z$ is totally semi-stable (on $\cohW$).
By \cite[Prop.~5.3]{B1}, it is a stability condition $\sigma=(Z,\cohW)$ on $\DQ$.
Clearly, $\sigma$ is also totally semi-stable.

If $\on{Im}(z)=0$, one can use Condition~$\bigstar$ and the same argument of the proof of Case \textbf{II} in \Cref{pp:heart}. It implies that
there is a heart $\h$ of the form $\on{mod}\CC \SEC^{\on{op}}$ such that $\h=\sli(1)$.
This implies that $\sigma$ is a concentrated stability condition. Since $\h$ is hereditary, $\sigma$ is semi-stable.
This completes the proof.
\end{proof}

In the following sections,
we will write down explicit inequalities for Condition~$\bigstar$ in Proposition~\ref{pp:AR} for each Euclidean case. %(For instance, in type $\Apq$, no extra condition is needed).

%=========================================================
%=========================================================
\section{The Euclidean case: affine type A}\
%=========================================================

For type $\Apq$, we have $\nw=2$ and $(w_1,w_2)=(p,q)$. The partitions \eqref{eq:part} become
\begin{equation}\label{eq:part-A}
\begin{cases}
  1=\mu_1^1+\cdots+\mu_1^p,\\
  1=\mu_2^1+\cdots+\mu_2^q.
\end{cases}
\end{equation}

\begin{figure}[tbh]  \centering
\begin{tikzpicture}[scale=1.5]
\foreach \x in {0,5}{
\begin{scope}[shift={(\x-.5,0)},yscale=1.25]
\pgfmathsetmacro{\sj}{.5}
\draw[white]
    (0,0) node[circle,draw=Emerald] (v0\x) {$00$}
    (1-\sj,1+\sj) node[circle,draw=Emerald] (v1\x) {$00$}
    (2-\sj,2+\sj) node[circle,draw=Emerald] (v2\x) {$00$}
    (3,3) node[circle,draw=Emerald] (v3\x) {$00$}
    (1.5+\sj+\sj,1.25) node[circle,draw=Emerald] (w\x) {$00$};
\end{scope}
\draw[-stealth,very thick,blue!60,opacity=.5]
    (v0\x)edge(v1\x) (v0\x)edge(w\x) (v1\x)edge(v2\x) (v2\x)edge(v3\x) (w\x)edge(v3\x);
}
\draw[-stealth,very thick,blue!60]
    (v05)edge(v15) (v05)edge(w5) (v15)edge(v25) (v25)edge(v35) (w5)edge(v35);

\draw[dashed, orange,stealth-](v00)edge(v05)(v10)edge(v15)(v20)edge(v25)(v30)edge(v35)(w0)edge(w5);
\draw[opacity=.5] (v00) node[]{$\tau P_0$} (v10) node[]{$\tau P_1^1$} (v20)node[]{$\tau P_1^2$}
    (v30)node[]{$\tau P_\infty$} (w0) node[]{$\tau P_2^1$};
\draw (v05) node[]{$P_0$} (v15) node[]{$P_1^1$} (v25)node[]{$P_1^2$}
    (v35)node[]{$P_\infty$} (w5) node[]{$P_2^1$};
\draw[stealth-,very thick,Emerald]
    (v05)edge(v10) (v05)edge(w0) (v15)edge(v20) (v25)edge(v30) (w5)edge(v30);

\begin{scope}[shift={(5,0)}]
\draw[font=\small,blue!60]
    (-.45,1)node{$\mu_1^1$} (.4,2.7)node{$\mu_1^2$} (1.65,3.6)node{$\mu_1^3$}
    (1,.7)node{$\mu_2^1$} (2.4,2.5)node{$\mu_2^2$};
\end{scope}
\draw[font=\small,blue!60,opacity=.5]
    (-.45,1)node{$\mu_1^3$} (.4,2.7)node{$\mu_1^1$} (1.65,3.6)node{$\mu_1^2$}
    (1,.7)node{$\mu_2^2$} (2.4,2.5)node{$\mu_2^1$};
\draw[font=\small,Emerald]
    (1.9,.9)node{$\mu_2^2$} (3,2.3)node{$\mu_2^2$} (4.5,3.6)node{$\mu_2^2$}
    (3.3,1.0)node{$\mu_1^3$} (4.6,2.5)node{$\mu_1^3$};
\end{tikzpicture}
\caption{Part of the AR-quiver of $\DQ\cong\D(\canA)$: type $\wh{A_{3,2}}$}\label{fig:AR-A}
\end{figure}

\begin{theorem}\label{thm:A}
To give a stability condition in $\PToSS(\Apq)$ is equivalent
to give a non-degenerate totally semi-stable datum $\TSD$.
\end{theorem}
\begin{proof}
%Further, we also know that any (indecomposable) object in $\on{coh}_0(\wpl)$ has phase 1 and
%the phase of any indecomposable vector bundle $X$ of $\wpl$ is
%$$\frac{1}{\pi}\arg Z(X) \in(0,1].$$
%Hence stability conditions in $\PToSS(\Apq)$ are determined by the associated totally semi-stable data.
%For the other side, we only need to check the criterion in \Cref{pp:AR}.
%Note that any totally semi-stable datum indeed determines a central charge $Z$ and
%the formula above defines the phase of any object $X$. What is left to show is that it is a stability conditions.
%It is sufficient to show that, for every arrow $X \to Y$ in the (shift of the) preprojective (or vector bundle) component $\on{AR}\on{vect}(\CC\PP^1_\wt)$ of the AR-quiver of $\DQ$, one has $\phi_\sigma(Y)>\phi_\sigma(X)$.
By \Cref{pp:AR}, we only need to check Condition~$\bigstar$ for a non-degenerate totally semi-stable datum $\TSD$.

For the arrows between projectives of the canonical algebra $\on{C}_\wt$,
we have 
\[
    Z([P_i^j])-Z([P_i^{j-1}])=Z([S_i^j])=-\mu_i^j<0, \  1\le i\le l, j\in\ZZ_{w_i}.
\]
As their phases are in $[0,1]$, we see that $\phi_\sigma(P_i^j)\ge\phi_\sigma(P_i^{j-1})$.
For any connecting arrows from $\Inj\canA[-1]$ to $\Proj\canA$, we have
\[\begin{cases}
    Z([P_1^{j-1}])-Z([\tau P_1^j])=Z([S_2^q])=-\mu_2^q<0,\\
    Z([P_2^{j-1}])-Z([\tau P_2^j])=Z([S_1^p])=-\mu_1^p<0.
  \end{cases}
\]
As above, we deduce that $\phi_\sigma(P_1^{j-1})\ge\phi_\sigma(\tau P_1^j)$ and $\phi_\sigma(P_2^{j-1})\ge\phi_\sigma(\tau P_2^j)$.

These are all the arrows up to $\tau$ in $\on{AR}\DQ$ and the calculations above are applicable under $\tau$ .
For instance, we have
\[
    Z([\tau^k P_i^j])-Z([\tau^k P_i^{j-1}])=Z([\tau^k S_i^j])=Z([S_i^{j+k}])=-\mu_i^{j+k}<0, \quad \forall k\in\ZZ.
\]
Thus, we are done.
\end{proof}

%Obviously, the partition \eqref{eq:part-A} exists and hence the non-degenerate totally semi-stable datum (of type $\Apq$) exists.

%=========================================================
%\subsection{Corollaries}
%%=========================================================
%
%
%We fix a total semi-stability condition $\sigma$ in with $Z(S_0)\in\CC_{>0}$.
\begin{corollary}\label{cor:con-A}
Given any non-concentrated stability condition $\sigma_0\in\PToSS(\Apq)$,
there is a linearly contractible flow on $\PToSS(\Apq)$ that contracts to $\sigma_0$.
In particular, $\ToSS(\Apq)$ is contractible.
\end{corollary}

\begin{proof}
Let $\TSD_0=(\lambda_?^?,z_0)$ denote the total semi-stability datum determined by $\sigma_0$. 
Since $\sigma$ is non-concentrated, then, by Proposition~\ref{pp:AR}, we have $\on{Im}(z_0)>0$.
Now we construct a flow on $\PToSS(\Apq)=\ToSS(\Apq)/\CC$ explicitly as follows.
Take any stability condition $\sigma\in\PToSS(\Apq)$ and $\TSD=(\mu_?^?,z)$ to be the corresponding total semi-stability datum determined by $\sigma$. Thus, we have $\on{Im}(z)\ge0$.
For any $t\in[0,1]$, define
\begin{equation}\label{eq:flow}
\begin{cases}
  \mu_i^j(t)=t\mu_i^j +(1-t) \lambda_i^j,& \quad 1\le i\le \nw,1\le j\le w_i. \\
  z(t)=tz+(1-t) z_0.
\end{cases}
\end{equation}
It is easy to check that each $\TSD(t):=(\mu_?^?(t),z(t))$ is a total semi-stability datum. Due to
the fact that 
\[
    \on{Im}z(t)=t\on{Im}(z)+(1-t)\on{Im}(z_0)>0, \quad\forall t\in[0,1),
\]
%following the argument of Case~\textbf{I} in Proposition~\ref{pp:heart}, we have $Z([V])\ne0$ for any indecomposable vector bundle $V\in\on{vect}(\wpl)$. Namely, 
each $\TSD(t)$ is non-degenerate. 
By Theorem~\ref{thm:A}, each $\TSD(t)$ defines a stability condition $\sigma(t)\in\PToSS(\Apq)$.
It is clear that $\sigma(0)=\sigma_0$ and $\sigma(1)=\sigma$. This completes the proof.
\end{proof}

\begin{corollary}
$\Stab\D_\infty(\Apq)$ is contractible.
\end{corollary}
\begin{proof}
By \cite{Q3}, we know that $\Stab\D_\infty(\Apq)$ contracts to $\ToSS(\Apq)$. 
Then the statement follows from \Cref{cor:con-A}.
\end{proof}

%=========================================================
%=========================================================
\section{The Euclidean case: affine type D}
%=========================================================
%=========================================================
%=========================================================

For type $\ADn$, we have $l=3$ and $(w_1,w_2,w_3)=(2,2,n-2)$. The partitions \eqref{eq:part} become
\begin{equation}\label{eq:part-D}
\begin{cases}
  1=\mu_1^1+\mu_1^2,\\
  1=\mu_2^1+\mu_2^2,\\
  1=\mu_3^1+\cdots+\mu_3^{n-2}.
\end{cases}
\end{equation}
One can find the tilting module for the canonical quiver \eqref{eq:pqr} into the AR-quiver as shown in \Cref{fig:AR-D} (the blue shaded circles for the $n=6$ case).
%=========================================================

\begin{figure}  \centering  \makebox[\textwidth][c]{
\begin{tikzpicture}[scale=1.9]
\pgfmathsetmacro{\hh}{.3}
\foreach \y in {-2+\hh-\hh,-2-\hh-\hh,-1,0,1,2+\hh-\hh,2-\hh-\hh}{
\draw[orange,dashed](-4,\y)to(4,\y);}
\foreach \x in {-4,-2,0,2}{
\begin{scope}[shift={(\x,0)}]
\draw[white,font=\tiny]
    (0,0) node[circle,draw=Emerald,fill=white] (v0) {$0$}
    (1,1) node[circle,draw=Emerald,fill=white] (v1) {$0$}
    (0,2+\hh-\hh) node[circle,draw=Emerald,fill=white] (w1) {$0$}
    (0,2-\hh-\hh) node[circle,draw=Emerald,fill=white] (u1) {$0$}
    (1,-1) node[circle,draw=Emerald,fill=white] (v2) {$0$}
    (0,-2+\hh-\hh) node[circle,draw=Emerald,fill=white] (w2) {$0$}
    (0,-2-\hh-\hh) node[circle,draw=Emerald,fill=white] (u2) {$0$};
\end{scope}
\draw[-stealth,ultra thick,Emerald,opacity=.69]
    (v0)edge(v1) (u1)edge(v1) (w1)edge(v1)
    (v0)edge(v2) (u2)edge(v2) (w2)edge(v2);}
\begin{scope}[xscale=-1]
\foreach \x in {-4,-2,0,2}{
\begin{scope}[shift={(\x,0)}]
\draw[white,font=\tiny]
    (0,0) node[circle,draw=Emerald,fill=white] (v0) {$0$}
    (1,1) node[circle,draw=Emerald,fill=white] (v1) {$0$}
    (0,2+\hh-\hh) node[circle,draw=Emerald,fill=white] (w1) {$0$}
    (0,2-\hh-\hh) node[circle,draw=Emerald,fill=white] (u1) {$0$}
    (1,-1) node[circle,draw=Emerald,fill=white] (v2) {$0$}
    (0,-2+\hh-\hh) node[circle,draw=Emerald,fill=white] (w2) {$0$}
    (0,-2-\hh-\hh) node[circle,draw=Emerald,fill=white] (u2) {$0$};
\end{scope}
\draw[stealth-,ultra thick,Emerald,opacity=.69]
    (v0)edge(v1) (u1)edge(v1) (w1)edge(v1)
    (v0)edge(v2) (u2)edge(v2) (w2)edge(v2);}
\end{scope}

\draw[cyan!20]
    (-4,-2+\hh-\hh) node[circle,draw=Emerald,fill=cyan!20](P0){$0$}
    (4,-2+\hh-\hh) node[circle,draw=Emerald,fill=cyan!20](P8){$0$}
    (0,-2+\hh-\hh) node[circle,draw=Emerald,fill=cyan!20](P4){$0$}
    (-2,-2-\hh-\hh) node[circle,draw=Emerald,fill=cyan!20](P2){$0$}
    (2,-2-\hh-\hh) node[circle,draw=Emerald,fill=cyan!20](P6){$0$}
    (0,2+\hh-\hh) node[circle,draw=Emerald,fill=cyan!20](P1){$0$}
    (0,2-\hh-\hh) node[circle,draw=Emerald,fill=cyan!20](P7){$0$};
\draw(P0)node{$P_0\,$}(P8)node{$P_\infty$}(P4)node{$P_3^2$}(P2)node{$P_3^1$}(P6)node{$P_3^3$}
    (P1)node{$P_1^1$}(P7)node{$P_2^1$}
    (-3,-1) node[gray]{$X_1$}         (-2.5-.15,-.5+.15)node[Emerald]{$\mu_3^2$}
    (-2,0) node[gray]{$X_2$}         (-1.5-.15,.5+.15)node[Emerald]{$\mu_3^3$}
    (-1,1) node[gray]{$X_3$}
    (-4,-2-\hh-\hh)node[gray]{$X_0$};
%    (-2,-2+\hh-\hh)node[gray]{$\tau P_3^2$};
\draw[-stealth,very thick,blue,opacity=.99]
    (P0)edge node[below]{$\mu_3^1$}(P2)
    (P2)edge node[below]{$\mu_3^2$}(P4)
    (P4)edge node[below]{$\mu_3^3$}(P6)
    (P6)edge node[below]{$\mu_3^4$}(P8)
    (P0)edge[bend left=13]  node[above]{$\mu_1^1$}(P1)
    (P1)edge[bend left=13] node[above]{$\mu_1^2$}(P8)
    (P0)edge[bend left=-10] node[below]{$\mu_2^1$}(P7)
    (P7)edge[bend left=-10] node[below]{$\mu_2^2$}(P8);
\end{tikzpicture}}
\caption{Part of the AR-quiver of $\D(\canA)$: type $\wh{D_6}$}\label{fig:AR-D}
\end{figure}

\begin{definition}\label{def:D}
A totally semi-stable datum $\TSD=(\mu_?^?,z)$  is of type $\ADn$ if
\begin{equation}\label{eq:D}
\begin{cases}
  |\mu_1^1-\mu_2^2|=|\mu_1^2-\mu_2^1|\le\mu_3^j,\\
  |\mu_1^1-\mu_2^1|=|\mu_1^2-\mu_2^2|\le\mu_3^j,
\end{cases}\quad 1\le j\le n-2 (=w_3)
\end{equation}.
\end{definition}

\begin{theorem}\label{thm:D}
To give a stability condition in $\PToSS(\ADn)$ is equivalent
to give a non-degenerate totally semi-stable datum of type $\ADn$.
Moreover, there exists a non-degenerate totally semi-stable datum of type $\ADn$.
\end{theorem}
\begin{proof}
Similar to type $\Apq$ case, we only need to check Condition~$\bigstar$ in \Cref{pp:AR}.
Consider the following section $\SEC$ of the AR-quiver:
\begin{equation}\label{eq:AR-D}
\begin{tikzcd}[column sep=20,row sep=10]
    P_0 \ar[dr] &&&& P_1^1\\
     & X_1 \ar[r] & \cdots \ar[r] & X_{n-3} \ar[dr]\ar[ur] \\
    X_0 \ar[ur] &&&& P_2^1.
\end{tikzcd}
\end{equation}
Using the dimension vectors (with respect to the quiver $\EQ=\SEC^{\on{op}}$),
one can calculate the central charges of objects in \eqref{eq:AR-D} explicitly (regarded as projectives of $\CC\EQ$):
\begin{equation}
\begin{cases}
  Z([P_0])=z, \\ Z([X_0])=z+1-\mu_1^1-\mu_2^1 \\%=z+\mu_2^2-\mu_1^1-\mu_3^1 \\
  \frac{1}{2}Z([X_i])=z+\frac{1}{2}(1-\mu_1^1-\mu_2^1-\sum_{j=1}^i \mu_3^j), & \mbox{$1\le i\le n-3$}  \\
  Z([P_1^1])=z-\mu_1^1,\\ Z([P_2^1])=z-\mu_2^1,
\end{cases}
\end{equation}
in terms of $z=Z([P_0])$ and the central charges $-\mu_?^?$ of some regular modules of $\EQ$.
Then we have
\begin{gather*}\begin{array}{rcl}
  \phi_\sigma(P_0)\le \phi_\sigma(X_1) &\Longleftrightarrow& \mu_2^2-\mu_1^1\le \mu_3^1 \\
  \phi_\sigma(X_0)\le \phi_\sigma(X_1) &\Longleftrightarrow& \mu_1^1-\mu_2^2\le \mu_3^1\\
  \phi_\sigma(X_{i-1})\le \phi_\sigma(X_i) &\Longleftrightarrow& \mu_3^i\ge 0,\qquad 2\le i\le n-3 \\
  \phi_\sigma(X_{n-3}) \le \phi_\sigma (P_1^1) &\Longleftrightarrow& \mu_2^1-\mu_1^1\le \mu_3^{n-2} \\
  \phi_\sigma(X_{n-3}) \le \phi_\sigma (P_2^1) &\Longleftrightarrow& \mu_1^1-\mu_2^1\le \mu_3^{n-2}
\end{array}\end{gather*}
Here we use the fact that $\on{Im}z\ge0$ and then the order of phases becomes
the reverse order of the real parts (provided that the imaginary parts coincide).

By applying $\tau$ (that shifts the superscripts $j$ of $\mu_i^j$), 
we can immediately write down all the inequalities for arrows in the $\tau$-orbits of the arrows above.
These arrows are precisely all arrows going (right-)upward in \Cref{fig:AR-D}.
Then applying the $\ZZ_2$-symmetry of $\widetilde{D}_n$ that interchanges the two rank two tubes
(i.e. $\mu_1^j \leftrightarrow \mu_2^j$),
we obtain all the inequalities for arrows going (right-)downward in \Cref{fig:AR-D}.
Together, we precisely get \eqref{eq:D},
which are precisely the condition of totally semi-stability.

Finally, \eqref{eq:D} holds if $\mu_1^1=\mu_1^2=\mu_2^1=\mu_2^2=1/2$.
Thus, there exists a non-degenerate totally semi-stable datum of type $\ADn$.
\end{proof}

%\begin{remark}[Convex polygon interpretation]
%\end{remark}
\begin{corollary}
Given any non-concentrated stability condition $\sigma_0$ in $\PToSS(\ADn)$,
there is contractible flow on $\PToSS(\ADn)$ that contracts to $\sigma_0$. 
In particular, $\ToSS(\ADn)$ is contractible.
\end{corollary}
\begin{proof}
The proof follows the same way as \Cref{cor:con-A},
using the same flow \eqref{eq:flow}.
The only thing to notice is that the linear combination formulae in \eqref{eq:flow}
ensures that linear inequalities in \eqref{eq:D} hold for
$\TSD(t)$ provided that they hold for $\TSD_0$ and $\TSD$.
\end{proof}

%=========================================================
%=========================================================
\section{The Euclidean case: affine type E}
%=========================================================

For type $\AEn (n=6,7,8)$, we have $l=3$ and $(w_1,w_2,w_3)=(2,3,n-3)$. Denote $r:=n-3$. The partitions \eqref{eq:part} become
\begin{equation}\label{eq:part-E}
\begin{cases}
  1=\mu_1^1+\mu_1^2,\\
  1=\mu_2^1+\mu_2^2+\mu_2^3,\\
  1=\mu_3^1+\cdots+\mu_3^{r}.
\end{cases}
\end{equation}
%=========================================================
One can embed the canonical quiver \eqref{eq:pqr} into the AR-quiver as shown in
\Cref{fig:AR-E6} for $\wh{E_6}$ and \Cref{fig:AR-E78} for $\wh{E_7}$ (top) and $\wh{E_8}$ (bottom) respectively.
These pictures are known back to Ringel \cite{Rin}. We redraw them with a modern take.

Similarly to affine type $D$, one can also illustrate 
Condition~$\bigstar$ in \Cref{pp:AR} to some inequalities/conditions for affine type $E$. 
For simplicity, we only list the corresponding inequalities/conditions for total semi-stability data
and omit the detailed calculations.

\begin{definition}\label{def:E}
A totally semi-stable datum $\TSD=(\mu_?^?,z)$ is of
\begin{itemize}
  \item type $\wh{E_6}$ if for any $i\in\ZZ_2, j,k\in\ZZ_3$
\begin{equation}\label{eq:E6}
\begin{cases}
\mu_1^i \le \mu_2^j+\mu_3^k,\\
    | (\mu_2^{j+1}-\mu_2^{j}) + (\mu_3^{k+1}-\mu_3^{k}) | \le \mu_1^i;
\end{cases}
\end{equation}
  \item type $\wh{E_7}$ if for any $i\in\ZZ_2,j\in\ZZ_3,k\in\ZZ_4$
\begin{equation}\label{eq:E7}
\begin{cases}
 \mu_1^i\le \mu_2^j+\mu_3^k,\\
  \mu_2^j\le \mu_3^{k-1}+\mu_3^{k+1},\\
   \mu_1^i+(\mu_2^j-\mu_2^{j\pm1})\le 2\mu_3^k+\mu_3^{k\pm1},\\
  \mu_1^i-\mu_1^{i+1}+(\mu_2^j-\mu_2^{j\pm1})\le 2\mu_3^k+\mu_3^{k\pm1}-\mu_3^{k\mp1};
\end{cases}
\end{equation}
  \item type $\wh{E_8}$ if for any $i\in\ZZ_2, j\in\ZZ_3, k\in\ZZ_5$
\begin{equation}\label{eq:E8}
\begin{cases}
\mu_2^j\le \mu_3^{k-1}+\mu_3^{k+1},\\
\mu_3^{k-1}+\mu_3^{k+1}\le \mu_1^i\le \mu_2^j+\mu_3^k,\\
\mu_1^i +(\mu_2^j-\mu_2^{j\pm1})\le 2\mu_3^k+\mu_3^{k\pm1},\\
\mu_2^j-\mu_2^{j\pm1}\le \mu_3^{k+1}+\mu_3^{k-1}-\mu_3^{k\pm2},\\
2\mu_1^i+(\mu_2^j-\mu_2^{j\pm2})\le
    \mu_3^{k\mp1}+2\mu_3^{k}+3\mu_3^{k\pm1},\\
\mu_3^{k\mp1}+2\mu_3^k+3\mu_3^{k\pm1}-\mu_3^{k\pm2}\le 3\mu_1^i+2(\mu_2^j-\mu_2^{j\mp1}),\\
(\mu_1^i-\mu_1^{i\pm1}) + (\mu_2^j-2\mu_2^{j\pm1})\le 
    2\mu_3^{k\mp2}+\mu_3^{k\mp1}-\mu_3^{k\pm1}-2\mu_3^{k\pm2}.
\end{cases}
\end{equation}
\end{itemize}
\end{definition}

\begin{theorem}\label{thm:E}
To give a stability condition in $\PToSS(\AEn)$ is equivalent
to give a non-degenerate totally semi-stable datum of type $\AEn$.
Moreover, there exists a non-degenerate totally semi-stable datum of type $\AEn$.
\end{theorem}

For the existence, one can take the obvious one: $\mu_i^j=1/w_i$ for any $1\le i\le 3, j\in\ZZ_{w_i}$.

\begin{corollary}
Given any non-concentrated stability condition $\sigma_0$ in $\PToSS(\AEn)$,
there is contractible flow on $\PToSS(\AEn)$ that contracts to $\sigma_0$. In particular, $\ToSS(\AEn)$ is contractible.
\end{corollary}

The rest of the section is the case-by-case proof of the theorem. In the following subsections, we will rewrite \eqref{eq:part-E} for affine type $E_n (n=6,7,8)$ via using the notation:
\[a_i=\mu_1^i,\ b_j=\mu_2^j,\ c_k=\mu_3^k,\quad\forall i\in\ZZ_2,j\in\ZZ_3,k\in\ZZ_{n-3},\]
for simplicity.
\begin{figure}[hbt]  \centering
\makebox[\textwidth][c]{
 \includegraphics[width=16cm]{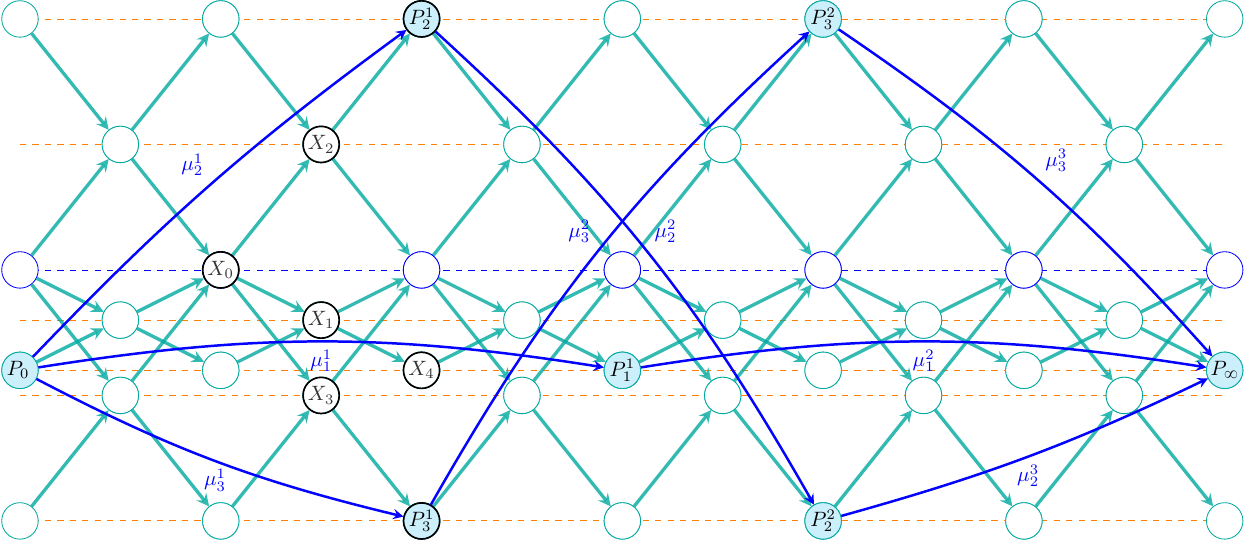}}
\caption{The canonical algebra quiver in the AR-quiver of $\D(\canA)$: type $\wh{E_6}$}\label{fig:AR-E6}
\end{figure}

%=========================================================
%\appendix
%=========================================================
%\section{Calculations for exceptional types}
%=========================================================
%=========================================================
\subsection{Proof for affine type $\widetilde{E_6}$}
%=========================================================
The partitions \eqref{eq:part-E} are rewritten as
\begin{equation}\label{eq:part-E6}
\begin{cases}
  1=a_1+a_2,\\
  1=b_1+b_2+b_3,\\
  1=c_1+c_2+c_3.
\end{cases}
\end{equation}
Take a section $\SEC$ of the AR-quiver (cf. \Cref{fig:AR-E6}):
\begin{equation}\label{eq:AR-E6}
\begin{tikzcd}[column sep=20,row sep=10]
    &X_2 \ar[r] & X_5 &&(X_5=P_2^1)\\
    X_0 \ar[r]\ar[dr]\ar[ur] & X_1 \ar[r] & X_4 \\
    &X_3 \ar[r] & X_6 &&(X_6=P_3^1)
\end{tikzcd}
\end{equation}
and we have the equivalence \eqref{eq:tri-equiv2} for $\EQ=\SEC^{ \on{op} }$.
Since the dimension vectors in each row of
\begin{equation}
\begin{pmatrix}
    \udim(X_0)\\
    \udim(S_1^1)\\
    \udim(S_1^2)\\
    \udim(S_2^1)\\
    \udim(S_2^3)\\
    \udim(S_3^1)\\
    \udim(S_3^3)
\end{pmatrix} =
 \begin{pmatrix}
   1&0&0&0&0&0&0 \\
   1&1&1&1&0&0&0 \\
   2&1&1&1&1&1&1 \\
   1&1&1&0&0&1&0 \\
   1&0&1&1&0&0&1 \\
   1&1&0&1&0&0&1 \\
   1&0&1&1&0&1&0 \\
 \end{pmatrix},
%\begin{array}{ccc}
%  \udim(S_1^1)=\,\begin{matrix} &1&0 \\ 1&1&0 \\ &1&0 \end{matrix} &
%  \udim(S_1^2)=\,\begin{matrix} &1&1 \\ 2&1&1 \\ &1&1 \end{matrix} &
%\\ [1cm]
%  \udim(S_2^1)=\,\begin{matrix} &1&1 \\ 1&1&0 \\ &0&0 \end{matrix} &
%  \udim(S_2^2)=\,\begin{matrix} &0&0 \\ 1&1&1 \\ &1&0 \end{matrix} &
%  \udim(S_2^3)=\,\begin{matrix} &1&0 \\ 1&0&0 \\ &1&1 \end{matrix}
%\\ [1cm]
%  \udim(S_3^1)=\,\begin{matrix} &0&0 \\ 1&1&0 \\ &1&1 \end{matrix} &
%  \udim(S_2^3)=\,\begin{matrix} &1&0 \\ 1&1&1 \\ &0&0 \end{matrix} &
%  \udim(S_3^2)=\,\begin{matrix} &1&1 \\ 1&0&0 \\ &1&0 \end{matrix} .
%\end{array}
\end{equation}
with respect to $\EQ$, are linearly independent, we know that
(the classes of) these objects form a basis of $K(\D)$
with central charges $z_0\colon=Z([X_0])$ and $\mu_i^j=-Z([S_i^j])$.
Further, one can calculate the following dimension vectors
\[
\begin{pmatrix}
    \udim( X_0)\\ \udim( \tau^{-1} X_0)\\
    \udim( X_1)\\ \udim( X_2)\\ \udim( X_3)\\
    \udim( X_4)\\ \udim( X_5)\\ \udim( X_6)\\
    \udim( \tau X_4)\\ \udim( \tau X_5)\\ \udim( \tau X_6)
\end{pmatrix} =
 \begin{pmatrix}
   1&0&0&0&0&0&0 \\
   2&1&1&1&0&0&0 \\
   1&1&0&0&0&0&0 \\
   1&0&1&0&0&0&0 \\
   1&0&0&1&0&0&0 \\
   1&1&0&0&1&0&0 \\
   1&0&1&0&0&1&0 \\
   1&0&0&1&0&0&1 \\
   0&0&0&0&-1&0&0 \\
   0&0&0&0&0&-1&0 \\
   0&0&0&0&0&0&-1 \\
 \end{pmatrix},
\]
and thus obtain
\[\begin{pmatrix}
    [X_0]\\ [\tau^{-1} X_0]\\
    3[X_1]\\3[X_2]\\3[X_3]\\
    3[X_4]\\3[X_5]\\3[X_6]\\
    3[\tau X_4]\\3[\tau X_5]\\3[\tau X_6]
\end{pmatrix}=
 \begin{pmatrix}
   1  & 0  & 0  & 0  & 0  & 0  & 0  \\
   1  & 1  & 0  & 0  & 0  & 0  & 0  \\
   2  & 1   & 0   & 1   & -1  & 1   & -1  \\
   2  & 1   & 0   & 1   & 2   & -2  & -1  \\
   2  & 1   & 0   & -2  & -1  & 1   & 2 \\
   1  & 2   & 3   & -1  & -2  & -1  & -2 \\
   1  & -1  & 0   & 2   & 1   & -1  & 1  \\
   1  & -1  & 0   & -1  & 1   & 2   & 1  \\
   1  & -1  & -3  & 2   & 1   & 2   & 1  \\
   1  & 2   & 0   & -1  & 1   & -1  & -2  \\
   1  & 2   & 0   & -1  & -2  & -1  & 1
 \end{pmatrix}\cdot
\begin{pmatrix}
  [X_0] \\ [S_1^1] \\ [S_1^2] \\ [S_2^1] \\ [S_2^3] \\ [S_3^1] \\ [S_3^3]
\end{pmatrix}.
\]
Considering the central charges and using relation \eqref{eq:part-E6}, we have the following,:
\begin{equation}\label{eq:E6-Z}
\begin{cases}
  Z([X_0])&=z_0, \\
  Z([\tau^{-1} X_0])&=z_0-a_1, \\

  3Z([X_1])&=2z_0-a_1-(b_1-b_3)-(c_1-c_3),     \\
  3Z([X_2])&=2z_0-a_1-(b_3-b_2)-(c_2-c_1),     \\
  3Z([X_3])&=2z_0-a_1-(b_2-b_1)-(c_3-c_2),     \\

  3Z([X_4])&=z_0-a_2-(b_2-b_3)-(c_2-c_3),     \\
  3Z([X_5])&=z_0-a_2-(b_1-b_2)-(c_3-c_1),     \\
  3Z([X_6])&=z_0-a_2-(b_3-b_1)-(c_1-c_2),     \\

  3Z([\tau X_4])&=z_0-(a_1-a_2)-(b_1-b_2)-(c_1-c_2),     \\
  3Z([\tau X_5])&=z_0-(a_1-a_2)-(b_3-b_1)-(c_2-c_3),     \\
  3Z([\tau X_6])&=z_0-(a_1-a_2)-(b_2-b_3)-(c_3-c_1).
\end{cases}
\end{equation}
For type $\widetilde{E_6}$, the semi-stability is equivalent to:
\begin{equation}\label{eq:E6-ssi}
\begin{cases}
    \phi_\sigma( \tau^t X_0)\le \phi_\sigma( \tau^t X_i)\le \phi_\sigma(\tau^{t-1} X_0),\\
    \phi_\sigma(\tau^{t+1} X_{i+3})\le \phi_\sigma(\tau^t X_i)\le \phi_\sigma( \tau^t X_{i+3}),
\end{cases}
\forall t\in\ZZ, i=1,2,3.
\end{equation}
For $t=0$, we have (using \eqref{eq:E6-Z} and eliminating $z_0$)
\begin{equation*}%\label{eq:E6<>}
\begin{cases}
  a_1\le b_2+c_2,  \\
  a_1\le b_1+c_3,  \\
  a_1\le b_3+c_1,  \\
  0\le a_1+(b_1-b_3)+(c_1-c_3)\le 2a_1,     \\
  0\le a_1+(b_3-b_2)+(c_2-c_1)\le 2a_1,     \\
  0\le a_1+(b_2-b_1)+(c_3-c_2)\le 2a_1.  
\end{cases}
\end{equation*}
By the $\tau$-symmetry that rotates $\{\mu_i^j\mid j\in\ZZ_{w_i}\}$,
one can obtain the following set of inequalities
\begin{equation*}%\label{eq:E6-abc}
\begin{cases}
  a_i\le b_j+c_k,   \\
    |(b_{j+1}-b_j) + (c_{k+1}-c_k) | \le a_i,   
\end{cases}   \forall i\in\ZZ_2, j,k\in\ZZ_3,
\end{equation*}
which is equivalent to \eqref{eq:E6-ssi} and is precisely \eqref{eq:E6} for type $\widetilde{E_6}$.

%=========================================================

%=========================================================
\subsection{Proof for affine type $\widetilde{E_7}$}
%=========================================================
The partitions \eqref{eq:part-E} are rewritten as
\begin{equation*}%\label{eq:part-E7}
\begin{cases}
  1=a_1+a_2,\\
  1=b_1+b_2+b_3,\\
  1=c_1+c_2+c_3+c_4.
\end{cases}
\end{equation*}
Take a section $\SEC$ of the AR-quiver (cf. the upper picture of \Cref{fig:AR-E78}):
\begin{equation}\label{eq:AR-E7}
\begin{tikzcd}[column sep=20,row sep=10]
    &X_3 \ar[r] & X_5 \ar[r] & X_7 &&(X_7=P_2^1)\\
    X_0 \ar[r]\ar[dr]\ar[ur] & X_1  \\
    &X_2 \ar[r] & X_4 \ar[r] & X_6 &&(X_6=\tau^{-1} P_3^1)
\end{tikzcd}
\end{equation}
and we have the equivalence~\eqref{eq:tri-equiv2} for $\EQ=\SEC^{ \on{op} }$.
Since the dimension vectors in each row of
\begin{equation}
\begin{pmatrix}
    \udim(X_0)\\
    \udim(S_1^1)\\
    \udim(S_1^2)\\
    \udim(S_2^1)\\
    \udim(S_2^2)\\
    \udim(S_3^1)\\
    \udim(S_3^2)\\
    \udim(S_3^3)
\end{pmatrix} =
 \begin{pmatrix}
   1&0&0&0&0&0&0&0 \\
   2&1&2&1&1&1&1&0 \\
   2&1&1&2&1&1&0&1 \\

   2&1&1&1&1&1&1&1 \\
   1&1&1&1&0&0&0&0 \\

   1&1&1&0&1&0&0&0  \\
   1&0&1&1&1&0&1&0  \\
   1&1&0&1&0&1&0&0
 \end{pmatrix},
%\begin{array}{ccc}
%  \udim(S_1^1)=\,\begin{matrix} &1&0 \\ 1&1&0 \\ &1&0 \end{matrix} &
%  \udim(S_1^2)=\,\begin{matrix} &1&1 \\ 2&1&1 \\ &1&1 \end{matrix} &
%\\ [1cm]
%  \udim(S_2^1)=\,\begin{matrix} &1&1 \\ 1&1&0 \\ &0&0 \end{matrix} &
%  \udim(S_2^2)=\,\begin{matrix} &0&0 \\ 1&1&1 \\ &1&0 \end{matrix} &
%  \udim(S_2^3)=\,\begin{matrix} &1&0 \\ 1&0&0 \\ &1&1 \end{matrix}
%\\ [1cm]
%  \udim(S_3^1)=\,\begin{matrix} &0&0 \\ 1&1&0 \\ &1&1 \end{matrix} &
%  \udim(S_2^3)=\,\begin{matrix} &1&0 \\ 1&1&1 \\ &0&0 \end{matrix} &
%  \udim(S_3^2)=\,\begin{matrix} &1&1 \\ 1&0&0 \\ &1&0 \end{matrix} .
%\end{array}
\end{equation}
with respect to $\EQ$, are linearly independent, we know that
(the classes of) these objects form a basis of $K(\D)$
with central charges $z_0\colon=Z([X_0])$ and $-\mu_i^j=Z([S_i^j])$.
Further, one can calculate the following dimension vectors
\[
\begin{pmatrix}
    \udim( X_0)\\ \udim( \tau^{-1} X_0)\\ \udim( X_1)\\
    \udim( X_2)\\ \udim( X_3)\\
    \udim( X_4)\\ \udim( X_5)\\ \udim( X_6)\\ \udim( X_7)\\
    \udim( \tau X_4)\\ \udim( \tau X_5)\\ \udim( \tau X_6)\\ \udim( \tau X_7)
\end{pmatrix} =
 \begin{pmatrix}
   1&0&0&0&0&0&0&0 \\
   2&1&1&1&0&0&0&0 \\
   1&1&0&0&0&0&0&0 \\
   1&0&1&0&0&0&0&0 \\
   1&0&0&1&0&0&0&0 \\
   1&0&1&0&1&0&0&0 \\
   1&0&0&1&0&1&0&0 \\
   1&0&1&0&1&0&1&0 \\
   1&0&0&1&0&1&0&1 \\
   0&0&0&0&-1&0&-1&0 \\
   0&0&0&0&0&-1&0&-1 \\
   0&0&0&0&0&0&-1&0 \\
   0&0&0&0&0&0&0&-1 \\
 \end{pmatrix},
\]
and thus obtain
\begin{equation}\label{eq:matrix7}
\begin{pmatrix}
    [X_0]\\ [\tau^{-1} X_0]\\
    2[X_1]\\ 4[X_2]\\ 4[X_3]\\
    2[X_4]\\ 2[X_5]\\ 4[X_6]\\ 4[X_7]\\
    2[\tau X_4]\\ 2[\tau X_5]\\ 4[\tau X_6]\\ 4[\tau X_7]
\end{pmatrix}=
 \begin{pmatrix}
    1&0&0&0&0&0&0&0  \\
    1&0&0&0&1&0&0&0  \\

    1&-1&-1&1&1&1&0&1  \\
    3&3&1&-1&1&-1&-2&-3  \\
    3&-1&1&-1&1&-1&2&1  \\

    1&1&1&-1&-1&1&0&-1  \\
    1&1&1&-1&-1&-1&0&1  \\
    1&1&-1&1&-1&1&2&-1  \\
    1&1&3&1&-1&-3&-2&-1  \\

    1&1&1&-1&1&-1&-2&-1  \\
    1&-1&-1&-1&1&1&2&1  \\
    1&1&3&-3&-1&1&-2&-1  \\
    1&1&-1&-3&-1&1&2&3  \\
 \end{pmatrix}\cdot
\begin{pmatrix}
  [X_0] \\ [S_1^1] \\ [S_1^2]
  \\ [S_2^1] \\ [S_2^2]
  \\ [S_3^1] \\ [S_3^2] \\ [S_3^3]
\end{pmatrix}.
\end{equation}
For type $\widetilde{E_7}$, the semi-stability is equivalent to:
\begin{equation}\label{eq:E7-t}
\begin{cases}
    \phi_\sigma( \tau^t X_0)\le \phi_\sigma(\tau^t X_i)\le \phi_\sigma(\tau^{t-1} X_0),\quad i=1,2,3,\\
    \phi_\sigma(\tau^{t+1} X_{i+2})\le \phi_\sigma(\tau^t X_{i} )\le \phi_\sigma(\tau^t X_{i+2})\quad i=2,3,4,5,
\end{cases} \forall t\in\ZZ.
\end{equation}
Similarly as before, one can express the central charges of objects in \eqref{eq:matrix7}
in terms of $z_0$ and $\mu_i^j$ and obtain (by eliminating $z_0$) a set of inequalities for $t=0$ in \eqref{eq:E7-t}. Applying the $\tau$-symmetry, one deduces that the semi-stability is equivalent to
\begin{equation*}%\label{eq:E7<>}
\begin{cases}
  a_i\le b_j+c_k,\\
  b_j\le c_{k-1}+c_{k+1},\\ 
  a_i+(b_j-b_{j\pm1})\le 2c_k+c_{k\pm1},\\
  a_{i}-a_{i+1}+(b_j-b_{j\pm1})\le 2c_k+c_{k\pm1}-c_{k\mp1},
\end{cases}\forall i\in\ZZ_2, j\in\ZZ_3, k\in\ZZ_4,
\end{equation*}
which is precisely \eqref{eq:E7} for type $\widetilde{E_7}$.
%=========================================================

%=========================================================
\subsection{Proof/Calculation for affine type $\widetilde{E_8}$}
%=========================================================
The partitions \eqref{eq:part-E} are rewritten as
\begin{equation*}%\label{eq:part-E8}
\begin{cases}
  1=a_1+a_2,\\
  1=b_1+b_2+b_3,\\
  1=c_1+c_2+c_3+c_4+c_5.
\end{cases}
\end{equation*}
Take a section $\SEC$ of the AR-quiver (cf. the lower picture of \Cref{fig:AR-E78}):
\begin{equation}\label{eq:AR-E8}
\begin{tikzcd}[column sep=20,row sep=10]
    &X_3 \ar[r] & X_5 \ar[r] & X_6 \ar[r] & X_7\ar[r] & X_8 && (X_8=\tau P_2^1)\\
    X_0 \ar[r]\ar[dr]\ar[ur] & X_1  \\
    &X_2 \ar[r] & X_4
\end{tikzcd}
\end{equation}
and we have the equivalence \eqref{eq:tri-equiv2} for $\EQ=\SEC^{ \on{op} }$.
Since the dimension vectors in each row of
\begin{equation}
\begin{pmatrix}
    \udim(X_0)\\
    \udim(S_1^1)\\
    \udim(S_1^2)\\
    \udim(S_2^1)\\
    \udim(S_2^2)\\
    \udim(S_3^1)\\
    \udim(S_3^2)\\
    \udim(S_3^3)\\
    \udim(S_3^4)
\end{pmatrix} =
 \begin{pmatrix}
   1&0&0&0&0&0&0&0&0 \\
   3&1&2&3&1&2&2&1&1\\
   3&2&2&2&1&2&1&1&0 \\

   2&1&2&1&1&1&1&0&0 \\
   2&1&1&2&1&1&1&1&0 \\

   1&0&1&1&1&1&0&0&0 \\
   1&1&0&1&0&1&1&0&0 \\
   1&0&1&1&0&1&1&1&0 \\
   2&1&1&1&1&1&1&1&1
 \end{pmatrix},
%\begin{array}{ccc}
%  \udim(S_1^1)=\,\begin{matrix} &1&0 \\ 1&1&0 \\ &1&0 \end{matrix} &
%  \udim(S_1^2)=\,\begin{matrix} &1&1 \\ 2&1&1 \\ &1&1 \end{matrix} &
%\\ [1cm]
%  \udim(S_2^1)=\,\begin{matrix} &1&1 \\ 1&1&0 \\ &0&0 \end{matrix} &
%  \udim(S_2^2)=\,\begin{matrix} &0&0 \\ 1&1&1 \\ &1&0 \end{matrix} &
%  \udim(S_2^3)=\,\begin{matrix} &1&0 \\ 1&0&0 \\ &1&1 \end{matrix}
%\\ [1cm]
%  \udim(S_3^1)=\,\begin{matrix} &0&0 \\ 1&1&0 \\ &1&1 \end{matrix} &
%  \udim(S_2^3)=\,\begin{matrix} &1&0 \\ 1&1&1 \\ &0&0 \end{matrix} &
%  \udim(S_3^2)=\,\begin{matrix} &1&1 \\ 1&0&0 \\ &1&0 \end{matrix} .
%\end{array}
\end{equation}
with respect to $\EQ$, are linearly independent, we know that
(the classes of) these objects form a basis of $K(\D)$
with central charges $z_0\colon=Z([X_0])$ and $-\mu_i^j=Z([S_i^j])$.
Further, one can calculate the following dimension vectors
\[
\begin{pmatrix}
    \udim( X_0)\\ \udim( \tau^{-1} X_0)\\ \udim( X_1)\\ \udim( X_2)\\ \udim( X_3)\\
    \udim( X_4)\\ \udim( X_5)\\ \udim( X_6)\\ \udim( X_7)\\ \udim( X_8)\\
    \udim( \tau X_4)\\ \udim( \tau X_5)\\ \udim( \tau X_6)\\ \udim( \tau X_7)\\ \udim( \tau X_8)
\end{pmatrix} =
 \begin{pmatrix}
   1&0&0&0&0&0&0&0&0 \\
   2&1&1&1&0&0&0&0&0 \\
   1&1&0&0&0&0&0&0&0 \\
   1&0&1&0&0&0&0&0&0 \\
   1&0&0&1&0&0&0&0&0 \\
   1&0&1&0&1&0&0&0&0 \\
   1&0&0&1&0&1&0&0&0 \\
   1&0&0&1&0&1&1&0&0 \\
   1&0&0&1&0&1&1&1&0 \\
   1&0&0&1&0&1&1&1&1 \\
   0&0&0&0&-1&0&0&0&0 \\
   0&0&0&0&0&-1&-1&-1&-1 \\
   0&0&0&0&0&0&-1&-1&-1 \\
   0&0&0&0&0&0&0&-1&-1 \\
   0&0&0&0&0&0&0&0&-1 \\
 \end{pmatrix},
\]
and thus obtain
\begin{equation}\label{eq:matrix8}
\begin{pmatrix}
    [X_0]\\ [\tau^{-1} X_0]\\
    2[X_1]\\ 3[X_2]\\ 6[X_3]\\
    3[X_4]\\ 3[X_5]\\ 2[X_6]\\ 3[X_7]\\ 6[X_8]\\
    3[\tau X_4]\\ 3[\tau X_5]\\ 2[\tau X_6]\\ 3[\tau X_7]\\ 6[\tau X_8]
\end{pmatrix}=
 \begin{pmatrix}
1&0&0&0&0&0&0&0&0 \\
1&1&1&0&0&-1&-1&-1&-1 \\

1&0&1&0&0&-1&0&-1&0 \\
2&1&1&1&-1&-1&-2&0&-1 \\
5&4&1&-2&2&-1&-2&-3&-4 \\

1&-1&-1&2&1&1&-1&0&1 \\
2&1&1&-2&-1&2&1&0&-1 \\
1&0&-1&0&0&1&2&1&0 \\
1&-1&-1&-1&1&1&2&3&1 \\
1&2&-1&-4&-2&1&2&3&4 \\

1&2&2&-1&-2&-2&-1&0&-2 \\
2&1&1&1&2&-1&-2&-3&-4 \\
1&0&1&0&0&1&0&-1&-2 \\
1&-1&-1&2&1&1&2&0&-2 \\
1&-4&-1&2&4&1&2&3&-2
 \end{pmatrix}\cdot
\begin{pmatrix}
  [X_0] \\ [S_1^1] \\ [S_1^2]
  \\ [S_2^1] \\ [S_2^2]
  \\ [S_3^1] \\ [S_3^2] \\ [S_3^3] \\ [S_3^4]
\end{pmatrix}
\end{equation}
For type $\widetilde{E_8}$, the semi-stability is equivalent to:
\begin{equation}\label{eq:E8-t}
\begin{cases}
    \phi_\sigma( \tau^t X_0)\le \phi_\sigma(\tau^t X_i)\le \phi_\sigma(\tau^{t-1} X_0),\quad i=1,2,3,\\
     \phi_\sigma(\tau^{t+1} X_{i+2})\le \phi_\sigma(\tau^t X_{i} )\le \phi_\sigma(\tau^t X_{i+2}),\quad i=2,3,\\
    \phi_\sigma(\tau^{t+1} X_{i+1})\le \phi_\sigma(\tau^t X_{i} )\le \phi_\sigma(\tau^t X_{i+1}),\quad i=5,6,7,\\
\end{cases}
\forall t\in\ZZ.
\end{equation}
Similarly as before, one can express the central charges of objects in \eqref{eq:matrix8}
in terms of $z_0$ and $\mu_i^j$ and obtain (by eliminating $z_0$) a set of inequalities for $t=0$ in \eqref{eq:E8-t}. Applying the $\tau$-symmetry, one deduces that the semi-stability is equivalent to
\begin{equation*}%\label{eq:E8<>}
\begin{cases}
b_j\le c_{k-1}+c_{k+1},\\
c_{k-1}+c_{k+1}\le a_i \le b_j+c_k,\\
a_i +(b_j-b_{j\pm1})\le 2c_k+c_{k\pm1},\\
b_j-b_{j\pm1}\le c_{k+1}+c_{k-1}-c_{k\pm2},\\
2a_i+(b_j-b_{j\pm2})\le c_{k\mp1}+2c_{k}+3c_{k\pm1},\\
c_{k\mp1}+2c_k+3c_{k\pm1}-c_{k\pm2}\le 3a_i+2(b_j-b_{j\mp1}),\\
(a_i-a_{i\pm1}) + (b_j-2b_{j\pm1})\le 2c_{k\mp2}+c_{k\mp1}-c_{k\pm1}-2c_{k\pm2},
\end{cases}\forall i\in\ZZ_2, j\in\ZZ_3, k\in\ZZ_5,
\end{equation*}
which is precisely \eqref{eq:E8} for type $\widetilde{E_8}$.

%=========================================================

%=========================================================

\[\]

\address{Qy:
	Yau Mathematical Sciences Center and Department of Mathematical Sciences,
	Tsinghua University,
    100084 Beijing,
    China.
    \&
    Beijing Institute of Mathematical Sciences and Applications, Yanqi Lake, Beijing, China}
\email{yu.qiu@bath.edu}

\address{Zx:
    Beijing Advanced Innovation Center for Imaging Theory and Technology, Academy for Multidisciplinary Studies, Capital Normal University, Beijing 100048, China}
\email{xiaoting.zhang09@hotmail.com}

%=========================================================
\begin{landscape}
%=========================================================
%=========================================================
%\section{The AR-quivers of $\D(\canA)$: type $\wh{E_7}$ and $\wh{E_8}$}
%=========================================================
\begin{figure}[hbt]  \centering \vskip .5cm
\makebox[\textwidth][c]{
 \includegraphics[width=25cm]{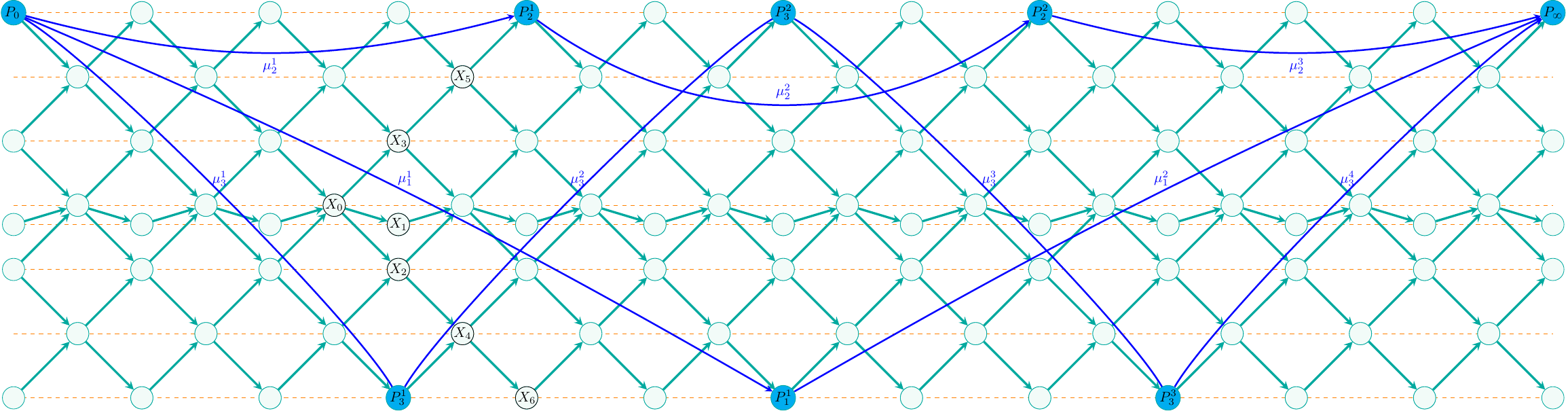} \qquad\qquad\qquad\qquad\qquad}
\vskip 1cm
\makebox[\textwidth][c]{
 \includegraphics[width=25cm]{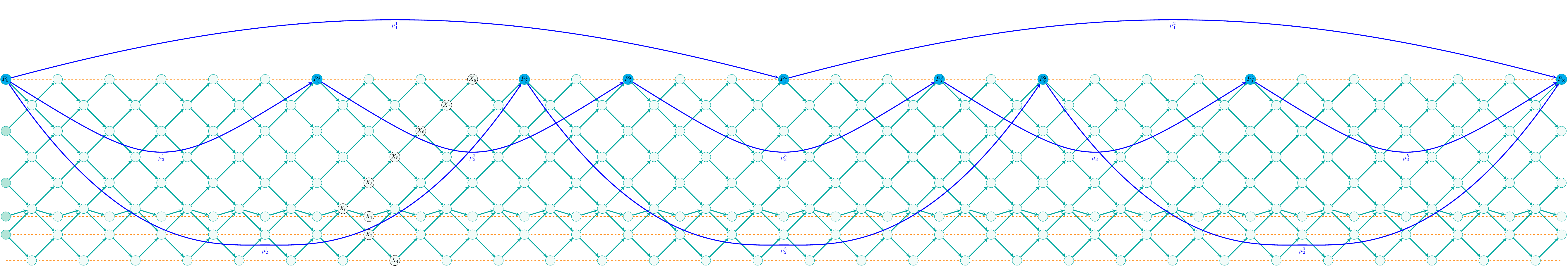} \qquad\qquad\qquad\qquad\qquad}
\vskip .5cm
\caption{The canonical algebra quivers in the AR-quivers of $\D(\canA)$: types $\wh{E_7}$ and $\wh{E_8}$}\label{fig:AR-E78}
\end{figure}
\end{landscape}

%=========================================================
\end{document}